\documentclass{article}
\usepackage{amssymb} 
 \usepackage{amsmath}
\usepackage{amsthm}
\makeatletter
\def\@seccntformat#1{\csname named#1\endcsname\csname the#1\endcsname.\ }
\makeatother

\usepackage{epsfig}
\usepackage{graphicx}
\usepackage{subfig} 

\theoremstyle{theorem}
\newtheorem{theorem}{Theorem}

\theoremstyle{definition}

    \def\ve#1{{\bf #1}}
    \def\norm#1{\| {\bf #1} \|}
   \def\normm#1{\| { #1} \|}
   \def\abs#1{\vert #1 \vert}
    \def\r{{\Bbb R}}

\begin{document}

\title{Length-Preserving Directions and Some Diophantine Equations}
\markright{Length-preserving directions}
\author{Juan Tolosa}

\maketitle

\begin{abstract}
We study directions along which the norms of vectors are preserved 
under a linear map. 
In particular, we find families of matrices for which these directions 
are determined by integer vectors. 
We consider the two-dimensional case in detail, 
and also discuss the extension to three-dimensional vector spaces.
\end{abstract}

\bigskip

\section{ Introduction.}

In the nice Webinar talk ``Eigenpairs in Maple'' 
of June 25, 2015 \cite{lopez}, 
Dr.~Robert Lopez discussed how to use {Maple} to 
find eigenvalues and eigenvectors (eigenpairs) of a
matrix $A$.  
An eigenvector of $A$ is a (nonzero) vector whose 
direction is preserved under multiplication by $A$.
By the end of the talk, Dr.~Lopez, as an aside, asked the question, 
what about preserving the {\it magnitude\/} of the vector, rather than its direction? 
In other words, what about (nonzero) vectors $\ve v$ such that
$\norm{v} = \normm{A \ve v}$, 
where $\norm{v}$ is the usual Euclidean norm? 
He provided the example of 

$$ A = \begin{pmatrix} 4&3\cr-2&-3  \end{pmatrix} ; $$
regarded as a map from $\r^2$ to itself;  
it preserves the norms, but not the directions,  of 
the vectors with integer coordinates 
$\ve v_1 = \langle 1, -1 \rangle$ and 
$\ve v_2 = \langle 17, -19 \rangle$. 
He clarified he had found such matrix by using {Maple} 
and ``for-loops'' to find 
matrices $A$ for which the equation
$\norm{ v} = \normm{A \ve v}$ would have integer solutions. 

In this article we explore this intriguing idea, 
obtain several families of such ``nice'' $2\times 2$ matrices, 
then
consider a few $3 \times 3$ examples, 
and discuss a couple of related quadratic Diophantine equations.

To avoid repetition, when considering an equation involving a vector $\ve v$, by a 
{\it solution\/} $\ve v$ 
we will consistently mean a {\it nontrivial solution\/} $\ve v \neq \ve 0$.

\section{General considerations.} 
A given $n \times n$-matrix $A$ with real entries generates  
a related linear map $\ve v \mapsto A \ve v$ of $\r^n$ into itself; 
we will use the same letter $A$ to represent this map.

We seek nonzero vectors $\ve v$ whose norm is preserved under 
this linear map; in other words, we seek (nonzero) solutions of the equation 
\begin{equation}
\norm{ v} = \normm{A \ve v}. 
\label{norma} 
\end{equation}

First of all, observe that, since $\norm{\lambda \ve v} = \abs{\lambda}\cdot\norm{\ve v}$, 
then the entire line generated by any nonzero solution of (\ref{norma}) 
will consist of vectors whose
norm is preserved by $A$; we will call these lines the {\it norm-preserving lines}.

Next, if $A$ has eigenvalue $1$, or $-1$, then the corresponding eigenspace
will consist entirely of solutions of (\ref{norma}) as well. 

The interesting case, though, is when there are nonzero solutions of (\ref{norma}) 
that are {\it not\/} eigenvectors. This may happen {\it even\/} if $A$ has an eigenvalue 
$\pm 1$. For example, the matrix
$$ 
A = 
\begin{pmatrix}  1&-8 \cr 0 & 3 \end{pmatrix} 
$$
has eigenvalue $1$ with eigenline determined by $\ve v = \langle 1, 0 \rangle$, 
but also has another, noninvariant, norm-preserving line,  determined by 
$\ve w = \langle 9, 2 \rangle$. 
Along this line, the map $A$ acts like a rotation.

At the other end of the spectrum we have the case of an orthogonal matrix $A$, for 
which {\it every\/} line through the origin is norm-preserving. In the $2\times 2$ case 
these are typically rotations, for which these lines are noninvariant; however, all lines 
are rotated by the same angle under $A$. In the general case this 
does not happen: each norm-preserving line is usually rotated by a different angle.

\section{The $2 \times 2$ case.} 

Let us discuss 
nonzero solutions of equation (\ref{norma}) 
for the case of a $2 \times 2$ real-valued matrix $A$. 
We will first obtain conditions for existence of such solutions, 
and next, find families of norm-preserving lines determined by integer vectors.

\subsection{Existence of a solution. }

Let us study solutions of equation (\ref{norma}) for a 
general real-valued $2 \times 2$-matrix
$$ A = \begin{pmatrix} a&b \\ c&d \end{pmatrix} . $$
Equation (\ref{norma}) is equivalent to 
$\norm{ v}^2 = \normm{A \ve v}^2$, or 
$(\ve v, \ve v) = (A \ve v, A \ve v)$, where
$ (\ve v, \ve w)$ is the usual Euclidean inner product. 
The right-hand side becomes
$$\normm{A \ve v}^2  = (A \ve v, A \ve v) 
     = (\ve v, A^t A \ve v) = (\ve v, B \ve v), $$
where $A^t$ is the transpose of $A$, and $B = A^t A$. 
Thus, equation (\ref{norma}) is equivalent to
\begin{equation} 
(\ve v, (B - I) \ve v) = 0 , 
\label{norma2} 
\end{equation}
Where $I$ is the identity matrix. 

Further, we have
\[
B = A^t \, A = \begin{pmatrix} a^2 + c^2 & ab+cd \\ ab+cd & b^2 +d^2 \end{pmatrix}
= \begin{pmatrix} m & p \\ p & n \end{pmatrix}, 
\]
where
\begin{eqnarray}
m & = & a^2 + c^2 ,  \label{i}\\
n & = & b^2 + d^2 , \label{ii} \\
p & = & ab + cd , \label{iii}
\end{eqnarray}
so that if we denote $\ve v = \langle x, y \rangle$, then the quadratic form at the left-hand side of 
(\ref{norma2}) is
\begin{equation}
\Phi(x, y) = (m-1)x^2 + 2p xy + (n-1) y^2 ; 
\label{forma}
\end{equation}
with this notation, (\ref{norma}) or, equivalently, (\ref{norma2}), is in turn equivalent to
$\Phi(x, y) = 0$.
We now prove the following result.

\begin{theorem}
Norm-preserving lines exist if and only if  
\begin{equation} a^2 + b^2 + c^2 + d^2 \ge 1 + \det(A)^2 . 
\label{condition}
\end{equation}
\end{theorem}

We will provide two proofs of this fact, one analytic, and one geometric.

\begin{proof}[First Proof]
Norm-preserving lines are determined by 
$\ve v = \langle x, y \rangle$, where $(x, y)$ is a nontrivial solution of  $\Phi(x, y) = 0$, 
where $\Phi(x, y)$ is given by (\ref{forma}).
If $n=1$, 
that is, if 
$b^2 + d^2 = 1$, this equation becomes 
$$ (m-1) x^2 + 2pxy = x \left[ (m-1) x + 2p y \right] = 0, $$
which has the nontrivial solution $\ve v = \langle 0, 1\rangle$. 
Also, in this case condition (\ref{condition}) is satisfied, since
\begin{align*}
&a^2 + b^2 + c^2 + d^2 - 1 - \det(A)^2  \cr
&= a^2 + c^2 - \det(A)^2 = (a^2+c^2) - (ad - bc)^2 ,
\end{align*}
which is equal to $(ab + cd)^2 = p^2$ under the assumption that 
$n = b^2 + d^2 = 1$. This follows from the identity
$$ (a^2+c^2) - (ad - bc)^2 - (ab + cd)^2 = [1- (b^2 + d^2)](a^2 + c^2) . $$
Thus, in this case condition (\ref{condition}) holds.

If $n \neq 1$, a nontrivial solution of 
$$ (m-1) x^2 + 2pxy + (n-1) y^2 = 0 $$
must satisfy $x \neq 0$. Dividing by $x^2$ and solving for $\frac yx$, we obtain
\begin{equation}
{y \over x} =  {- p \pm \sqrt{p^2 - (m-1)(n-1)} \over n-1} . 
\label{eqForRatio}
\end{equation}
A solution will exist if and only if the discriminant
$ p^2 - (m-1)(n-1) $
is non-negative.
Further,
$$ p^2 - (m-1)(n-1) = (p^2 - mn) + m + n - 1, $$
and 
$$ mn - p^2 = \det(B) = \det(A)^2 = (ad - bc)^2 , $$
which can also be checked directly using (\ref{i})--(\ref{iii}). 
Thus, there is a norm-preserving line if  and only if  
$$ m + n - 1 - \det(A)^2 =
(a^2+c^2)+(b^2+d^2) - 1 -\det(A)^2 \ge 0 , $$
which coincides with condition (\ref{condition}).
\end{proof}

\begin{proof}
[Second Proof]
The eigenvalues $\lambda_1$ and $\lambda_2$ of 
the symmetric matrix $B = A^t A$ are (real and) nonnegative;
assume $0 \le \lambda_1 \le \lambda_2$. 
By the extreme properties of eigenvalues (see, for example, \cite{gelfand}
or \cite{shilov}), we have 
\begin{eqnarray*}
\lambda_1  &=& \min_{\norm{v} = 1} \normm{A \ve v}^2 
= \min_{\norm{v} = 1} (\ve v, B \ve v)    \cr
  &\le& \max_{\norm{v} = 1} \normm{A \ve v}^2 
= \max_{\norm{v} = 1} (\ve v, B \ve v) = \lambda_2 .  
\end{eqnarray*}

Therefore, there will exist norm-preserving lines $\ve v$ 
such that $\norm{v} = \normm{A \ve v}$ if  and only if  
$$ \lambda_1 \le 1 \le \lambda_2. $$
When the eigenvalues are strictly positive, this condition guarantees
the intersection of the ellipse 
$(\ve v, B \ve v) = 1$, 
for which the half-axes are 
$$ {1 \over \sqrt{\lambda_1}} \qquad \hbox{ and } \qquad  {1 \over \sqrt{\lambda_2}}, $$
with the unit circle $ \norm{v} = 1$ (see Figure \ref{firstEllipse}).

\begin{figure}[htbp]
\begin{center}
\includegraphics[height=50mm]{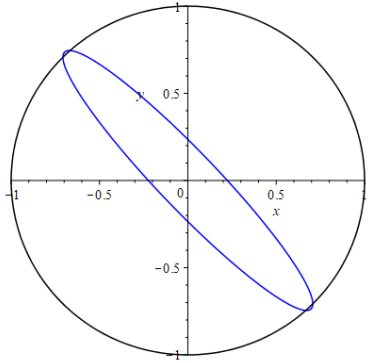}
\caption{Illustration of R. Lopez's example.}
\label{firstEllipse}
\end{center}
\end{figure}

The eigenvalues of $B$ are found from the characteristic equation
\begin{equation}
\lambda^2 - t \lambda + \Delta = 0, 
\label{chareq}
\end{equation}
where
$$ t = {\rm trace}\, {B} \qquad\hbox{ and }\qquad \Delta = \det B . $$
In terms of $B$, condition (\ref{condition}) reads as
$$ t \ge 1 + \Delta . $$
Notice that $\Delta = \det(B) \ge 0$, so actually (\ref{condition}) 
implies $t \ge 1 + \Delta \ge 1 > 0$. 
Solving (\ref{chareq}) we find
$$ \lambda = {t \over 2} \pm {1\over2} \sqrt{t^2 - 4 \Delta} . $$
Let us assume now that  $t \ge 1 + \Delta$. Then
$$ t^2 - 4 \Delta \ge (1 +\Delta)^2 - 4 \Delta = (1 - \Delta)^2 . $$
so that $\sqrt{t^2 - 4 \Delta} \ge \abs{1 - \Delta}$. 
For the largest eigenvalue $\lambda_1$ we have
$$ \lambda_1 \ge {1 \over 2}(1 + \Delta) + {1 \over 2}\abs{1- \Delta} . $$
Considering both cases $1 \ge \Delta$ and $\Delta \ge 1$, we conclude 
that 
\begin{equation}
\lambda_1 \ge \max\{ \Delta, 1\} . 
\label{boundOne}
\end{equation}
In particular, we have $\lambda_1\ge 1$, as desired.
Now let $\lambda_2$ be the smallest eigenvalue. 
Since $\lambda_1 \lambda_2 = \Delta$, we have
$$ \lambda_2 = {\Delta \over \lambda_1} , $$
and condition (\ref{boundOne}) now implies that indeed
$\lambda_2 \le 1$. 
Thus, condition (\ref{condition}) guarantees the existence of norm-preserving lines. 
The converse is straightforward. 
\end{proof}

\subsection{Families of matrices with integer solutions.}

We want to find matrices
$$  A = \begin{pmatrix} a&b \\ c&d \end{pmatrix}  $$
with integer or rational entries, for which the norm-preserving lines are determined by
vectors $\ve v$ with integer coordinates; we will call these  
{\it integer solution lines.}%
      \footnote{Of course, if, say, a norm-preserving line is determined 
      by $\ve v = \langle 1, 2 \rangle$, then it is also determined by
      $\ve v = \langle \sqrt{2}, 2 \sqrt{2} \rangle$. The idea is that {\it there exists\/} a determining
      vector with integer coordinates. In the $2 \times 2$  case, 
      this means we are interested in solution lines with rational slopes.}
Recall that 
$$\normm{A \ve v}^2 =  (A\ve v, A \ve v) = (\ve v, B \ve v), $$
where
\[
B = A^t \, A = \begin{pmatrix} a^2 + c^2 & ab+cd \\ ab+cd & b^2 +d^2 \end{pmatrix}
= \begin{pmatrix} m & p \\ p & n \end{pmatrix}, 
\]
so that the solutions of $\normm{A \ve v}^2 = \norm{\ve v}^2$
are given by $\ve v = \langle x, y \rangle$, where $(x, y)$ is a (nontrivial) solution of 
\begin{equation}
 (m-1) x^2 + 2p xy + (n-1) y^2 = 0 . 
 \label{diophantine1}
 \end{equation}
If either $m=1$ or $n=1$, it is not hard to get families of solutions.
For example, for the two-parameter family 
$$ A = \begin{pmatrix} \frac3{\strut 5} & b \cr \frac{\strut 4}5 & d \end{pmatrix} , $$
we get norm-preserving lines determined by 
$\ve v_1 = \langle 1, 0 \rangle$ and 
$$ \ve v_2 = \langle 5(1 - b^2 - d^2), 2(3b+4d) \rangle . $$

\medskip

For the general case, let us assume that the entries of $A$ are integers, and 
seek integer solutions of the Diophantine equation (\ref{diophantine1}). 
Assuming $n \neq 1$ and solving for $y/x$, as we did in 
(\ref{eqForRatio}), we conclude that there are integer solution lines if  and only if 
the discriminant $p^2 - (m-1)(n-1)$ is a perfect square, say, $k^2$. 
This leads to the new Diophantine equation 
$$ p^2 - (m-1)(n-1) = k^2 , $$
which can be rewritten as 
 \begin{equation}
 (m-1)(n-1) = p^2 - k^2 = (p+k)(p-k) ,
 \label{diophantine2}
 \end{equation}
where $m, n, p$ are given by (\ref{i})--(\ref{iii}). 

We will now find two-parameter families of solutions.
To this end, let us set 
\begin{eqnarray*}
m-1 &=& p+k, \\
n-1 &=& p-k,
\end{eqnarray*}
or 
\begin{eqnarray}
(a^2+c^2)-1 &=& p+k,  \label{eq1} \\
(b^2+d^2)-1 &=& p-k,   \label{eq2}
\end{eqnarray}
to which we will 
include the equation 
\begin{equation}
2(ab + cd) = 2p , \label{eq3}
\end{equation}
stemming from the definition (\ref{iii}) of $p$.
Adding (\ref{eq1}) and (\ref{eq2}), and subtracting (\ref{eq3}), we get 
\begin{equation}
(a-b)^2 + (c-d)^2 = 2,  
\label{eq4}
\end{equation} 
from which we conclude that
\begin{equation}
\abs{a-b} = 1 \qquad \hbox { and } \qquad \abs{c-d} = 1. 
\label{absoluto}
\end{equation}
Next,
subtracting (\ref{eq2}) from (\ref{eq1}), we obtain
\begin{equation*}
a^2 - b^2 + c^2 - d^2  = 2k,
\end{equation*}
or
\begin{equation}
(a-b)(a+b) + (c-d)(c+d) = 2k
\label{eq5}
\end{equation}
Considering all the possibilities in (\ref{absoluto}), 
we obtain the following four two-parameter families of matrices
$$ \begin{pmatrix} a & a \pm 1 \cr c & c \pm 1 \end{pmatrix} . $$
For each of them, we can use (\ref{eq5}) to find the value of $k$.
Incidentally,  their transposes, 
$$ \begin{pmatrix} a & c \cr a \pm 1 & c \pm 1 \end{pmatrix} , $$
also have integer solution lines. 

The example of Robert Lopez corresponds to the case 
$a-b = 1$ and $c - d = 1$, which leads to 
\begin{equation}
A = \begin{pmatrix} a & a - 1 \cr c & c - 1 \end{pmatrix} ; 
\label{lopezGen}
\end{equation}
he chose the particular values 
$a = 4$ and $c = -2$. 
For the general solution of type (\ref{lopezGen}), 
we find from (\ref{eq5}) that $k = a + c - 1$. 
Substituting this value into (\ref{eqForRatio}), which now looks like
$$ 
{y \over x} =  {- p \pm k \over n-1} ,  $$
we conclude that
the norm-preserving lines are determined by the vectors 
$$ \ve v_1 = \langle 1, -1 \rangle $$
and 
$$ \ve v_2 = \langle (a-1)^2+(c-1)^2-1, 1- a^2 - c^2 \rangle . $$
The remaining cases can be discussed similarly.

\begin{figure}[htbp]
\begin{center}
\includegraphics[height=50mm]{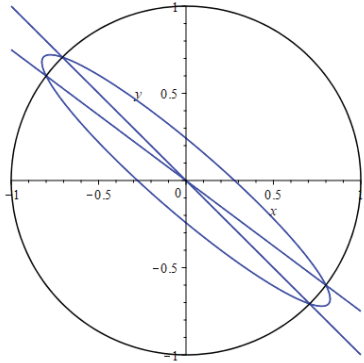}
\caption{Norm-preserving directions shown.}
\label{secondEllipse}
\end{center}
\end{figure}

Figure \ref{secondEllipse} illustrates the case $a=2$, $c=-3$;
the two length-preserving lines are shown. 
The direction vectors are $\ve v_1 = \langle 1, -1 \rangle$ and $\ve v_2 =  \langle 4, -3 \rangle$. 
As expected, 
the solution lines pass through the intersections of the ellipse 
$\normm{A \ve v} = 1$ with the unit circle. 

The corresponding picture for the matrix chosen by Robert Lopez was 
shown in Figure \ref{firstEllipse}; the solutions lines are not depicted, 
since they are rather close to each other.

\smallskip

If the ellipse is tangent to the unit circle, as for example when $a = 3$, $c = -2$,
we get only one integer solution line.

\section{The $3 \times 3$ case.} 

The $3 \times 3$ case is considerably more complicated, as well as more interesting.

If $A$ is a $3 \times 3$ real-valued matrix, also regarded as a linear map from $\r^3$ to
itself, then in general 
$ \normm{A \ve v}^2 = 1$ is an ellipsoid, in terms of the coordinates
of $\ve v = \langle x, y, z \rangle$. 
As in the $2\times2$ case, we have
$$  \normm{A \ve v}^2 = (A \ve v, A \ve v) = (\ve v, B \ve v ) , $$
where 
$B = A^t A$ is a symmetric matrix with nonnegative eigenvalues. 
The equation for norm-preserving vectors, $\normm{A \ve v} = \norm{v}$, 
or $(\ve v, B \ve v) =( \ve v, \ve v)$, 
is equivalent to the cone
\begin{equation}
(\ve v, (B-I)\ve v) = 0 , 
\label{thecone}
\end{equation}
where $I$ is the identity matrix. 

As in the $2 \times 2$ case, if we denote the eigenvalues of $B$ by
$0 \le \lambda_1 \le \lambda_2 \le \lambda_3$, 
then there is a solution of (\ref{thecone}) if  and only if 
\begin{equation}
\label{conditionFor3by3}
\lambda_1 \le 1 \le \lambda_3 , 
\end{equation}
which guarantees a nonempty intersection of the ellipsoid (or degenerate ellipsoid) 
$\normm{B \ve v}^2 = 1$ 
with the unit sphere $\norm{v}^2 = 1$.

If condition 
(\ref{conditionFor3by3}) 
is satisfied, then the cone determined by (\ref{thecone}) will pass through this intersection of the ellipsoid and the unit sphere.
As an illustration, for the matrix in Example 1,  
Figure \ref{ExampleOne} (a) shows the ellipsoid 
$\normm{A \ve v}^2 = 1$ and the 
unit sphere, and Figure \ref{ExampleOne} (b) has 
the added solution cone $\normm{A \ve v} = \norm v$;  
compare with Figure \ref{secondEllipse} for the $2 \times 2$ case.

\begin{figure}[ht]
\centering
\subfloat[]{{\includegraphics[height=53mm]{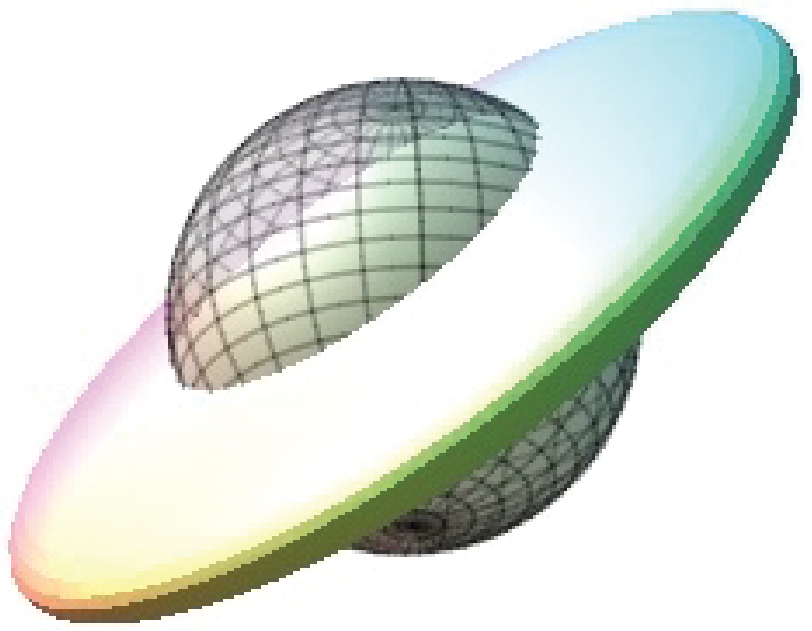}}}
\quad
\subfloat[]{{\includegraphics[height=53mm]{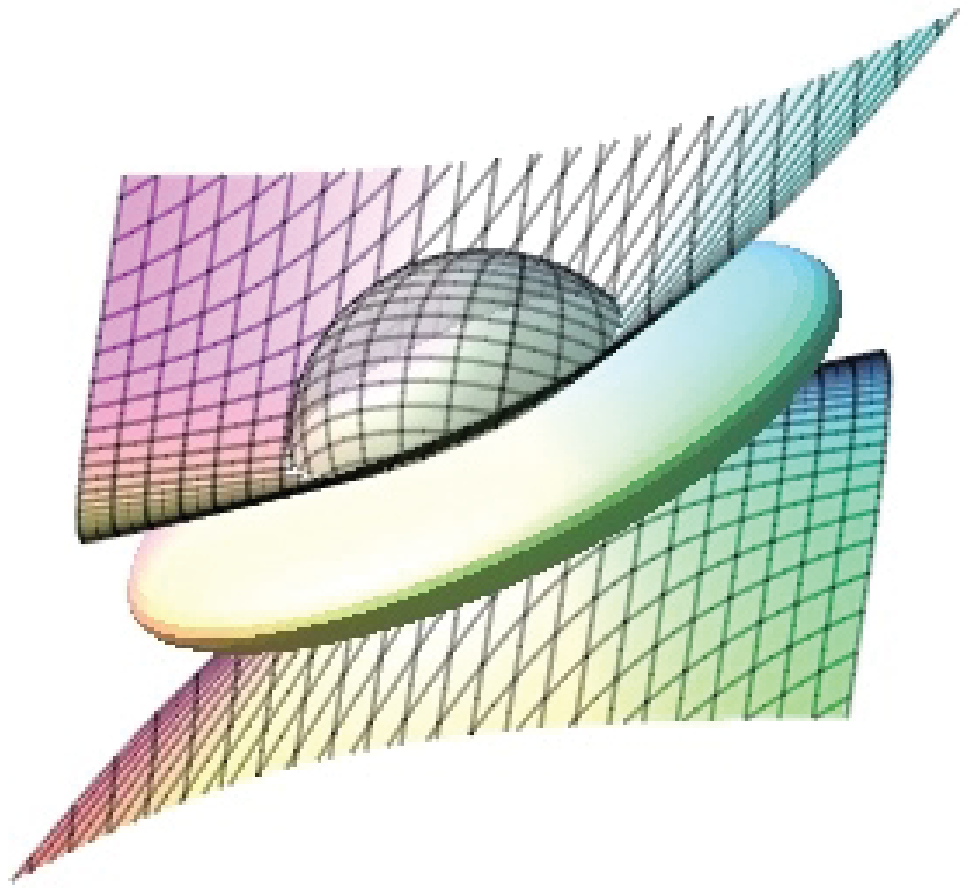}}}
\caption{Illustration for Example 1: (a) ellipsoid and sphere; 
(b) same, plus solution cone.
\label{ExampleOne}}
\end{figure}

For the $3 \times 3$ case, however, it is considerably harder to find an expression 
for condition (\ref{conditionFor3by3}) directly in terms of the entries of $A$.
Instead, 
we will limit ourselves to discuss several examples that exhibit the 
various possible 
outcomes regarding the existence of integer solution lines.

Before giving examples, let us add another comment: 
unlike the $2 \times 2$ case, 
in the $3 \times 3$ case one cannot hope that in general {\it all\/} the lines in the cone (\ref{thecone}) 
for a given matrix $A$ with integer or rational coefficients  
will turn out to be integer solution lines. 
Examples 2 and 3 show that, however, we can still get infinitely many such lines. 

\subsection{Example 1: no integer solution lines.} 

Let us consider the symmetric matrix
\begin{equation}
      A = \begin{pmatrix}  1&1&\frac{\strut 1}{\strut 2}\cr
                                      1&\frac{\strut 1}{\strut 2}&1\cr
                                      \frac{\strut 1}{\strut 2}&1&1 \end{pmatrix}
             \label{firstA}
\end{equation}

The form
$ \normm{A \ve v}^2 =  (A \ve v, A \ve v) = (\ve v, B \ve v) $
will in this case have matrix
$$ B = A^t A = A^2 = 
\begin{pmatrix}  \frac{\strut 9}{\strut 4}&2&2\cr
                                      2& \frac{\strut 9}{\strut 4}&2\cr
                                      2&2& \frac{\strut 9}{\strut 4} \end{pmatrix}
$$
so that equation (\ref{thecone}) will be
\begin{equation}
\frac 54 x^2 + 4 x y + 4 x z + \frac54  y^2 + 4 y z + \frac54 z^2 = 0 . 
\label{thisCone}
\end{equation}
This equation has infinitely many real-valued solutions, which 
constitute the solution cone shown in Figure \ref{ExampleOne} (b).
We want to show, however, that there are no integer solution lines, that is,
no nontrivial 
vectors $\ve v$ with integer coordinates such that $\normm{A \ve v} = \norm{v}$.

Indeed, (\ref{thisCone}) 
is a quadratic equation in $z$, which we can solve:
$$ z = - {8 \over 5}(x+y) \pm {\sqrt{39 x^2 + 48 xy + 39 y^2} \over 5} .
$$
Therefore, there will be integer solution lines if  and only if  
the discriminant $39 x^2 + 48 xy + 39 y^2$ is a perfect square.
We now prove this  does not happen. 

\begin{theorem}
\label{theo2}
The Diophantine equation
\begin{equation} 39 x^2 + 48 x y  + 39 y^2 = u^2 \label{uno} 
\end{equation}
has no nontrivial solutions, that is, 
no nonzero integer solutions.
\end{theorem}

\begin{proof}
Let $x = 2^a v$ and $y = 2^b w$, with $a, b$ nonnegative integers  
and $v, w$ odd. 
We may assume that $a \le b$; otherwise, we interchange $x$ and $y$. 
Then
$$ 39 x^2 + 48 x y  + 39 y^2 = 
2^{2a}(39 v^2 + 2^{b-a} 48 \, vw + 2^{2(b-a)} 39 \, w^2) , $$
and this will be a square only if  and only if  
the expression in parentheses is a square. 
But if $b = a$ this expression is congruent to $2$ mod 4, 
while if $b > a$ this expression is congruent to $3$ mod 4,
so in either case this is impossible. 
\end{proof}

\subsection{Example 2: a dense set of integer solution lines.} 

Consider the symmetric matrix
\begin{equation}
      A = \begin{pmatrix}  1&2&2\cr
                                      2&1&2\cr
                                      2&2&1 \end{pmatrix} .
             \label{firstA}
\end{equation} 
Here
$$ B = A^t A = A^2 = 
\begin{pmatrix}              9&8&8\cr
                                      8&9&8\cr
                                      8&8&9 \end{pmatrix}
$$
so that equation (\ref{thecone}) will be
\begin{equation*}
8x^2 + 16 x y + 16 x z + 8 y^2 + 16 y z + 8 z^2 = 0 , 
\end{equation*}
or (after dividing by $8$)
\begin{equation}
(x+y+z)^2  = 0 . 
\label{aPlane}
\end{equation}
In our case, the cone degenerates into the plane $x+y+z = 0$.
We can pick an integer basis, say
$\ve v_1 = \langle 1, 0, -1 \rangle$
and
$\ve v_2 = \langle 0, 1, -1 \rangle$,
and obtain every integer norm-preserving line as generated by
$\alpha \ve v_1 + \beta \ve v_2$, with integer coefficients 
$\alpha, \beta$ such that $\alpha^2 + \beta^2 > 0$.  
This constitutes a dense set of integer solution lines, among all possible
solutions in the plane $x+y+z=0$.

The ellipsoid $\normm{A \ve v}^2 = 1$ lies inside the unit sphere, 
and is tangent to it along the intersection of the sphere with the solution plane; 
Figure \ref{EllipseAndPlane} depicts the situation. 

\begin{figure}[htbp]
\begin{center}
\includegraphics[height=50mm]{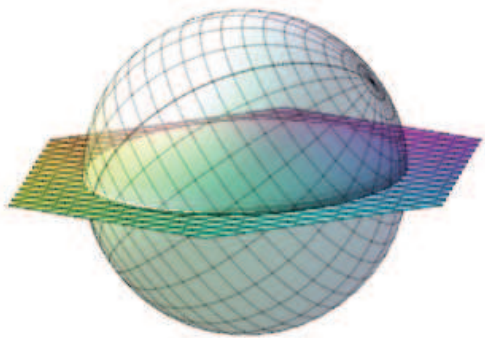}
\caption{Solution cone is degenerate.}
\label{EllipseAndPlane}
\end{center}
\end{figure}

\subsection{Example 3: infinitely many integer solution lines.} 

Finally, let us consider an example when there are still infinitely many integer 
solution lines, yet we cannot guarantee they are dense in the cone of all 
solution lines. 
Consider the matrix

\begin{equation}
      A = \begin{pmatrix}  1&2&3\cr
                                      2&1&1\cr
                                      1&1&1 \end{pmatrix} .
             \label{firstA}
\end{equation} 
One eigenvalue of $A$ is $-1$, with eigenvector $\ve = \langle -1, 1, 0 \rangle$, 
which therefore provides one integer solution line.
Are there any other such lines? The matrix $B = A^t A$ is
$$ B =  
\begin{pmatrix}              6&5&6\cr
                                      5&6&8\cr
                                      6&8&11 \end{pmatrix}
$$
and consequently  
equation (\ref{thecone}) becomes
\begin{equation*}
5 x^2 + 10 x y + 12 x z + 5 y^2 + 16 y z + 10 z^2 = 0 .
\end{equation*}
Solving for $x$ (which provides a slightly shorter answer than solving for $z$) yields
$$ x = - y - {6 \over 5} z \pm {\sqrt{-20 yz - 14 z^2} \over 5} . $$
We conclude that there will be integer solutions lines if  and only if  the 
discriminant $-20 yz - 14 z^2$ is a perfect square.%
             \footnote{Notice, by the way, that by setting $z=0$ we get back the
             eigenline determined by $\langle -1, 1, 0 \rangle$.}
This leads to the Diophantine equation
\begin{equation}
-20 yz - 14 z^2 = u^2 . 
\label{diophantine3}
\end{equation}

Let us find all integer solutions of (\ref{diophantine3}). 
As remarked previously, if $z=0$ we get back the eigenline 
corresponding to $\lambda = -1$, 
generated by $\langle -1, 1, 0\rangle$, 
so let us assume $z \neq 0$. 

If $z$ is odd, then 
$$ -20yz - 14z^2 \equiv -14 \equiv 2 \pmod 4 , $$
and hence cannot be a square. 
Therefore  $z$ must be even, so 
$z = 2a$ for some integer $a$. 
Substituting into (\ref{diophantine3}) we get
$$ 8 ( - 5 y a - 7 a^2) = u^2 , $$
whence $u$ is divisible by $4$. 
Letting $u = 4b$ and dividing by $8$, we obtain
$$ a(-5y - 7a) = 2b^2 . $$
Letting 
$a = v^2 q$, with $q$ square-free, we obtain
\begin{equation}
\label{intermediate}
v^2 q ( -5y - 7v^2 q )  = 2 b^2 . 
\end{equation}
Denoting for a moment $p = -5y - 7 v^2 q$, the above equation reads
\begin{equation}
\label{intermediate2}
v^2 pq = 2 b^2 , 
\end{equation}
where $q$ is square-free.
This implies that $pq$ is even.

If $p$ is even, so $p = 2t$ for some integer $t$, then (\ref{intermediate2}) becomes
$v^2 tq = b^2$, and it follows that $t = qr^2$ for some integer $r$, so that 
\begin{equation}
\label{pIs}
p = 2q r^2.
\end{equation}
On the other hand, if $p$ is odd, then we must have 
$q = 2 s$, with $s$ odd; hence (\ref{intermediate2}) is now
$v^2 ps = b^2$, whence 
$p = s r^2$ for some integer $r$, and therefore
$$ p = s r^2 = {q \over 2} r^2 = 2 q \left({r \over 2}\right)^2 , $$
and we get back expression (\ref{pIs}) for $p$, 
if we allow $r$ to be an integer or a half-integer. 

We conclude that $-5y - 7 v^2 q$ must be equal to $2qr^2$, for some integer 
or half-integer $r$. 
Solving for $y$, we obtain 
\begin{equation}
\label{yValue}
y = - {q \over 5} (7v^2 + 2r^2) , 
\end{equation}
where $v, q$, and $2r$ are arbitrary integers, subject only to the condition that 
$y$ be an integer. 
Moreover, since rational coordinate vectors also determine integer solution lines, 
we can even drop this additional condition.

To find the value of the other variables in terms of $(v, q, r)$, 
we observe first that 
\begin{equation}
\label{zValue}
z = 2a = 2v^2 q, 
\end{equation}
and that, 
from (\ref{intermediate}), we get 
$ b = vqr , $
whence
\begin{equation}
\label{uValue}
u = 4vqr . 
\end{equation}
(We have chosen only the ``$+$'' sign for $b$, since the ``$\pm$'' is recovered 
when computing the $x$-value; see (\ref{xValue}).)
Next, since 
\begin{equation}
\label{xValue}
x = -y - {6 \over 5} z  \pm {1 \over 5} u , 
\end{equation}
we get
\begin{equation}
x = \displaystyle {q \over \strut 5}(2 r^2 -5 v^2 \pm 4vr) .
\end{equation} 
Equations (\ref{xValue}), (\ref{yValue}), and (\ref{zValue}) provide 
the coordinates of all possible integer solution lines. 
Observing, however, that $q$ is a common factor of all three coordinates, 
and proportional vectors provide the same solutions,
we may just set $q=1$, 
replace $r$ by $r/2$,  and 
conclude
that all solution lines 
of (\ref{diophantine3}) are given by
$$
\begin{cases} x = \displaystyle {1 \over \strut 10}(r^2 - 10 v^2 \pm 4vr) \cr 
                      y =- \displaystyle {1 \over \strut 10} (14 v^2 + r^2) \cr
                      z = 2 v^2 , \cr
\end{cases} $$
where $v$ and $r$ are arbitrary integers. 

For example, choosing $(v, r) = (1, 4)$ we get the vectors
$$ \left\langle {11 \over 5}, -3, 2 \right\rangle  \sim \langle 11, -15, 10 \rangle 
             \qquad \hbox{ and } \qquad \langle 1, 3, -2 \rangle ; $$
and choosing  $(v, r) = (1, 1)$ we obtain
$$ \left\langle  -{1\over2},  -{3\over2}, 2  \right\rangle \sim \langle 1, 3, -4 \rangle 
\qquad \hbox{ and } \qquad \left\langle  - {13 \over 10}, - {3 \over 2}, 2 \right\rangle 
  \sim \langle 13, 15, -20 \rangle . $$

\subsection{A general method.} 
There is a method for obtaining infinitely many integer solution lines, which 
can be applied to a general $3 \times 3$ matrix
with integer or rational coefficients; 
the drawback is that one must know (at least) one nontrivial solution. 
This is an idea by  T.~Piezas \cite{wolfram}. 
Namely, if we know one particular integer solution 
$(y, z, u) = (m, n,  p)$
of the Diophantine equation
$$ a  y^2 + b  y  z + c  z^2= d  u^2, $$
then a two-parameter family of solutions is given by
 \begin{eqnarray*}
 y &=& (a m+b n) s^2 + 2 c n s t- c m t^2, \\
z &=&a n s^2+2 a m s t+(b m+c n) t^2, \\
u &=& p (a s^2+b s t+c t^2), 
\end{eqnarray*}
where $s$ and $t$ are arbitrary integers. 

For example, for the matrix
$$  A = \begin{pmatrix} 1&2&3\cr3&4&5\cr2&3&4  \end{pmatrix} , $$
equation (\ref{thecone}) for $\ve v = \langle x, y, z \rangle$ is
$$ 13 x^2 + 40 xy + 52 xz + 28 y^2 + 76 yz + 49 z^2 = 0 , $$
which solved with respect to $x$ produces
\begin{equation}
\label{ValueOfX}
x = {- 20 y \over 13} - 2z \pm {\sqrt{36y^2 + 52 yz + 39 z^2} \over 13} ,
\end{equation}
so to get integer solutions we need to solve the Diophantine equation
\begin{equation}
36y^2 + 52 yz + 39 z^2 = u^2.
\end{equation}
Setting $z=0$, it is not hard to guess the particular solution 
$(y, z, u) = (m, n, p) = (1, 0, 6)$. 
The corresponding two-parameter solution family looks like
\begin{eqnarray}
y &=& 36s^2 - 39 t^2\cr
z &=& 72 st + 52 t^2\cr
u &=& 216 s^2 + 312 st + 234 t^2.
\end{eqnarray}
Each choice of $(s, t)$ produces two integer solution lines $\ve v = \langle x,y,z \rangle$,
using the two $x$-values provided by (\ref{ValueOfX}).
For example, for $(s, t) = (1,1)$ we get
$$ \left\langle - {2402 \over 13}, -3, 124 \right\rangle \sim \langle -2402, -39, 1612\rangle, 
\qquad\hbox{ and } \qquad \langle -302, -3, 124 \rangle ; $$
and $(s, t) = (1, 2)$ yields
$$ \left\langle - {4976 \over 13}, -120, 352 \right\rangle \sim \langle -4976, -1560, 4576\rangle, 
\qquad\hbox{ and } \qquad \langle -656, -120, 352 \rangle . $$

\section{Questions for further study, applications.} 

\begin{itemize}

\item 
Unlike the case of eigenpairs, the solution lines to equation (\ref{norma}) 
are very much dependent on the chosen norm in $\r^n$. It would be interesting 
to discuss similar solutions for other norms in Euclidean space. 

\item
We have only grazed the case of a $3 \times 3$ matrix $A$. Can we find solvability conditions 
of (\ref{norma}) in terms of the coefficients of $A$, as we did in the $2 \times 2$ case? 
Can one find nontrivial families of integer matrices $A$ for 
which (\ref{norma}) has integer solutions? 

\item 
{\bf Application to toral automorphisms.} 
Many of the $2 \times 2$ integer matrices we studied, for example, 
the subfamily  
\begin{equation}
\label{autom}
A = \begin{pmatrix} q+1 & q \cr q & q-1 \end{pmatrix} , 
\end{equation}
with $q$ integer, 
are symmetric and have determinant $-1$; 
therefore, they can be regarded also as linear automorphisms of the (flat) 2-torus, 
which possess very interesting dynamical properties;
see, for example \cite{katok}, p.~42. 
Nontrivial such automorphisms have eigenvectors with irrational slopes; 
on the other hand, the integer solution lines of (\ref{autom}) bisect the eigendirections. 
We can therefore use integer arithmetic to compute iterates of vectors in the stable
and in the unstable manifolds of such automorphisms. 
For example, the matrix
\begin{equation}
\label{automPart}
A = \begin{pmatrix} 3 & 2 \cr 2 & 1 \end{pmatrix} 
\end{equation}
has integer solution lines generated by 
$\ve v_1 = \langle 1, -1 \rangle$ and $\ve v_2 =\langle -1, 3 \rangle$, 
and irrational eigenvalues
$$ \lambda_1 = 2+\sqrt 5 \qquad\hbox{ and } \qquad \lambda_2 =  2 - \sqrt 5 . $$
If we choose the equal-norm vectors
$$ \ve v_1 =  \langle 1, -1 \rangle \qquad\hbox{ and } \qquad
    \ve v_3 = \frac{1}{\sqrt 5}  \ve v_2 = \frac{1}{\sqrt 5} \langle -1, 3 \rangle , $$
then $\ve u = \ve v_1 + \ve v_3$ will be along the unstable direction of $A$, and 
$\ve w = \ve v_1 - \ve v_3$ will be along the stable direction. 
Therefore, on the one hand
$$ A^n \ve u = \lambda_1^n \ve u = (1 + \sqrt 5)^n \ve u, $$
and on the other hand
$$ A^n \ve u = A^n \ve v_1 + \frac{1}{\sqrt{5}} A^n \ve v_2 . $$
For example, 
\begin{align*}
&A^{10} \ve u = \begin{pmatrix}
1346269 & 832040 \cr 832040 & 514229
\end{pmatrix}  \ve u =   \cr
&= \left\langle 514229 + \frac{1}{\sqrt{5}} 1149851,  \,
317811 + \frac{710647}{\sqrt{5}}  \right\rangle ,
\end{align*}
provides a way to compute $(2 + \sqrt 5)^{10} \ve u$ 
using only integer arithmetic. 
A similar calculation can be used for iterates of vectors in the stable direction. 

\end{itemize}

\section{Acknowledgments.} 
%
I am grateful to Dr.~Robert Lopez for introducing the interesting idea of
vectors whose norms are preserved by a linear map. 
I also wish to thank the Editorial Board member of the {\it American Mathematical Monthly},
for valuable comments, corrections, and suggestions 
that were used to significantly improve the overall quality of the paper. 
This applies, especially, to a shorter and more conceptual proof of Theorem \ref{theo2}, 
and to a complete solution of the Diophantine equation (\ref{diophantine3}).

%
%
%
%
\vfill\eject

\end{document}